\documentclass[a4paper,oneside,reqno,12pt]{amsart}

\usepackage[utf8]{inputenc}
\usepackage[T1]{fontenc}
\usepackage[english]{babel}
\usepackage{lmodern}
\usepackage{amsmath,amssymb,amsfonts,mathtools}

\usepackage[
  backend=biber,
  style=numeric,
  sorting=none,
  giveninits=true,
  maxbibnames=99,
  doi=true,
  url=true,
  eprint=true
]{biblatex}

\usepackage{amsmath,amssymb,amsthm}
\usepackage{mathtools}
\usepackage{xcolor}
\usepackage{graphicx}
\usepackage{tikz}
\usetikzlibrary{arrows.meta,calc,positioning}
\usepackage[colorlinks=true,linkcolor=blue,citecolor=blue,urlcolor=blue]{hyperref}

\DeclareNameAlias{sortname}{given-family}

\renewbibmacro{in:}{}

\DeclareFieldFormat[article]{title}{#1}

\usepackage{fullpage}
\usepackage{hyperref}
\newcommand{\secref}[1]{\hyperref[#1]{\textcolor{red}{\ref*{#1}}}}
\hypersetup{
  colorlinks=true,
  linkcolor=blue,
  citecolor=blue,
  urlcolor=blue
}

\newtheorem{theorem}{Theorem}
\newtheorem{lemma}[theorem]{Lemma}
\newtheorem{corollary}[theorem]{Corollary}

\newtheorem{problem}[theorem]{Problem}
\theoremstyle{definition}
\newtheorem{definition}[theorem]{Definition}
\theoremstyle{remark}

\DeclareMathOperator{\conv}{conv}
\DeclareMathOperator{\aff}{aff}

\DeclareMathOperator{\Span}{span}
\DeclareMathOperator{\cone}{cone}

\title[Colorful Carath\'eodory for spanning trees]{Spanning \(k\)-trees and the colorful Carath\'eodory theorem }

\author{Mikhail Bludov and Alexander Polyanskii}
\address{Mikhail Bludov, Moscow Institute of Physics and Technology, Dolgoprudny, Moscow region, Russia }
\email{\href{mailto:bludov.mv@phystech.edu}{bludov.mv@phystech.edu}}
\address{Alexander Polyanskii, Department of Mathematics, Emory University, 400 Dowman Drive, Atlanta, GA 30322, USA }
\email{\href{mailto:apolian@emory.edu}{apolian@emory.edu}}
\urladdr{\url{https://polyanskii.com}}

\begin{document}

\begin{abstract}
Very recently, using Meshulam's lemma, Blagojevi\'c proved a constrained version of the colorful Carath\'eodory theorem for joins of bipartite spanning trees and wedge of spheres. Our main contribution extends his result from joins of bipartite spanning trees with wedges of spheres to joins of spanning \(k\)-trees with wedges of spheres. Our proof is elementary and avoids the topological machinery. We also discuss a homological variation of spanning \(k\)-trees and some Carath\'eodory-type results for them.
\end{abstract}
\maketitle

\section{Introduction}

The colorful Carath\'eodory theorem of B\'ar\'any~\cite{Barany1982} is one of the
fundamental results in discrete geometry. It states that if
\(X_1,\ldots,X_{d+1}\) are finite subsets of \(\mathbb R^d\) such that the origin
\(0\) lies in \(\conv X_i\) for every \(i\), then there are points
\(x_i\in X_i\) for all \(i\), such that
\(0\in \conv\{x_1,\ldots,x_{d+1}\}\). We refer the interested reader to the recent
survey~\cite[Section 3.1]{DeLoeraGoaocMeunierMustafa2019} for recent progress, as well
as connections to other classical theorems in the field.

In this paper, we use the standard simplicial-complex formulation of the colorful
Carath\'eodory theorem. We assume familiarity with basic notions such as abstract
simplicial complexes, geometric realizations, faces and their dimensions, joins, and
\(d\)-spheres. To keep the introduction light while making the paper self-contained,
we recall the relevant definitions in Section~\ref{section notation}.

Let \(K\) be an abstract simplicial complex, and $|K|$ be its geometric realization. We say that a map
\(A:|K|\to\mathbb R^d\) is \textit{affine} if it is determined for vertices of $|K|$ and then affinely extended for all faces of $|K|$. We use the following notational conversion. For a face \(\sigma\in K\), we write
\(A(\sigma):=A(|\sigma|^\circ)\), where $|\sigma|^\circ$ is the open simplex of the geometric realization of $\sigma$. Similarly, for any $L\subset K$, we write $A(L):=\cup_{\sigma\in L} A(\sigma)$.

Throughout this section, let \(V_1,\ldots,V_{d+1}\) be pairwise disjoint finite sets,
and put \(V=V_1\sqcup\cdots\sqcup V_{d+1}\). Let \(\Delta_V:=2^V\) denote the abstract simplex on
the vertex set \(V\). When finite sets appear in a join, we regard them as \(0\)-dimensional simplicial
complexes, whose only non-empty faces are the one-point subsets. Thus
\(V_1*\cdots *V_{d+1}\) denotes the complex of all subsets of
\(V_1\sqcup\cdots\sqcup V_{d+1}\) containing at most one vertex from each \(V_i\).

We say that an affine map \(A:|\Delta_V|\to\mathbb R^d\) is a \emph{colorful
Carath\'eodory map} if \(0\in A(V_i)\) for every
\(i\). With this terminology, the colorful Carath\'eodory theorem can be
reformulated as follows.

\begin{theorem}[Colorful Carath\'eodory theorem]
\label{theorem colorful caratheodory}
If \(A:|\Delta_V|\to\mathbb R^d\) is a colorful Carath\'eodory map, then there exists a
face \(\sigma\in V_1*\cdots *V_{d+1}\) such that \(0\in A(\sigma)\).
\end{theorem}

Very recently, Blagojevi\'c~\cite[Problem 1.5]{Blag} proposed a research problem. One special case can be formulated as follows.

\begin{problem}
Describe all minimal subcomplexes \(K\subseteq V_1*\cdots *V_{d+1}\) with the
following property: for every colorful Carath\'eodory map \(A:|\Delta_V|\to\mathbb R^d\),
there exists a face \(\sigma\in K\) such that \(0\in A(\sigma)\).
\end{problem}

Towards this problem, Blagojevi\'c recalled to the readers that the following variation holds, which is a special case of the very colorful Carath\'eodory theorem; see~\cite{HolmsenPachTverberg2008,ArochaBaranyBrachoFabilaMontejano2009}.
\begin{theorem}
\label{theorem k=0}    
If \(A:|\Delta_V|\to\mathbb R^d\) is a colorful Carath\'eodory map and $v_1\in V_1$ is a vertex, then there exists a
face \(\sigma\in \{v_1\}*V_2*\cdots *V_{d+1}\) such that \(0\in A(\sigma)\).
\end{theorem}
Moreover, Blagojevi\'c proved the following version of the
colorful Carath\'eodory theorem. A subcomplex \(T\subseteq V_1*V_2\) is called a \emph{spanning tree} if it contains the empty set and all
vertices of \(V_1\sqcup V_2\), and its one-dimensional faces form the edge set of a
spanning tree of the complete bipartite graph with parts \(V_1\) and \(V_2\).

\begin{theorem}[Blagojevi\'c]
\label{theorem blagojevic k=1} 
Let \(T\subseteq V_1*V_2\) be a spanning tree. If \(A:|\Delta_V|\to\mathbb R^d\) is a
colorful Carath\'eodory map, then there exists a face
\(\sigma\in T*V_3*\cdots *V_{d+1}\) such that \(0\in A(\sigma)\).
\end{theorem}

The proof of this theorem relies on the study of the so-called zero-avoiding complexes,
their Alexander duals, and Meshulam's lemma. The reader can find a presentation of
this lemma and its proof in~\cite[Proposition 2.6]{DeLoeraGoaocMeunierMustafa2019};
see also Appendix~\ref{appendix meshulam proof}, where we use it to give an
alternative proof of our main result. Blagojevi\'c's proof appears
to be quite technical. The aim of this note is to show an extension of Theorems~\ref{theorem k=0}  and~\ref{theorem blagojevic k=1} by elementary means.

Indeed, the one-vertex complex from Theorem~\ref{theorem k=0} and the spanning tree from
Theorem~\ref{theorem blagojevic k=1} are the first two instances of the following
higher-dimensional notion. We leave the corresponding verification to the reader as a simple exercise. 

\begin{definition}
\label{definition k-tree}
Let \(0\le k\le d\). We call a subcomplex
\(T\subseteq V_1*\cdots *V_{k+1}\) a \emph{spanning \(k\)-tree} if it contains the full
\((k-1)\)-skeleton of \(V_1*\cdots *V_{k+1}\), that is, all faces of dimension at most
\(k-1\), and satisfies the following two conditions:
\begin{enumerate}
    \item \(T\) contains no $k$-spheres;
    \item $T$ is maximal with respect to (1), that is, for every facet \(\sigma\) of \(V_1*\cdots *V_{k+1}\)  that does not belong to \(T\), the complex \(T\cup\{\sigma\}\) contains a unique $k$-sphere.
\end{enumerate}
\end{definition}
We prove the following theorem.

\begin{theorem}
\label{theorem main}
Let \(T\subseteq V_1*\cdots *V_{k+1}\) be a spanning
\(k\)-tree for some $0\leq k\leq d$. If
\(A:|\Delta_V|\to\mathbb R^d\) is a colorful Carath\'eodory map, then there exists a face
\(\sigma\in T*V_{k+2}*\dots *V_{d+1}\) such that \(0\in A(\sigma)\).
\end{theorem}
Unlike the proof of Blagojevi\'c, our argument is elementary and based on an adaptation of the original idea of B\'ar\'any to consider the nearest-to-origin simplex. In Appendix~\ref{appendix meshulam proof}, we give a second proof based on Meshulam's lemma and the zero-avoiding complex. In fact, that argument yields a more general sufficient condition on the constraint complex. \smallskip

The paper is organized as follows. In Section~\secref{section notation} we recall the basic definitions, standard facts, and notational conventions used throughout the paper. We also formulate an elementary auxiliary lemma used in the proof of Theorem~\ref{theorem main}. In Section~\secref{section proof} we prove Theorem~\ref{theorem main}. In Section~\secref{generalized k-trees} we recall the necessary terminology and facts from elementary algebraic topology. Also, we introduce \(\mathbb Z_2\)-spanning \(k\)-trees, which generalize the spanning \(k\)-trees from Definition~\ref{definition k-tree}, and prove the corresponding generalization of Theorem~\ref{theorem main}. In Section~\secref{section discussion} we compare spanning \(k\)-trees with \(\mathbb Z_2\)-spanning \(k\)-trees and discuss applications of the preceding results to several Carath\'eodory-type theorems. Finally, Appendix~\ref{appendix meshulam proof} presents the alternative proof via Meshulam's lemma.


\section{Notation and auxiliary facts}
\label{section notation}

Let us briefly recall the standard notation. An \textit{abstract simplicial complex} \(K\) is a family of finite subsets, called \textit{faces}, of some ground set, whose elements are called \textit{vertices}, such that \(K\) is closed under taking subsets: if \(\sigma\in K\) and \(\tau\subseteq \sigma\), then \(\tau\in K\). The dimension of a face \(\sigma\in K\) is the size of $\sigma$ minus $1$.  The dimension of a complex is the maximum dimension of its faces. A face \(\sigma\) of a simplicial complex \(K\) is called a \emph{facet} or \textit{maximal face} if it is inclusion-maximal, that is, if there is no face \(\tau\in K\) such that \(\sigma\subsetneq \tau\). 

If \(K\) and \(L\) are abstract simplicial complexes on disjoint vertex sets, their
\textit{join} is the simplicial complex
\(
        K*L:=\{\sigma\cup\tau:\sigma\in K,\ \tau\in L\}.
\)
The join operation is associative, in the sense that
\(
        K*(L*M)=(K*L)*M
\)
for complexes on pairwise disjoint vertex sets. Thus we shall simply write
\(K*L*M\), and similarly for longer joins.

Let \(K\) be an abstract simplicial complex on a finite vertex set \(V\). To realize
\(K\) geometrically, place each vertex \(v\in V\) at the corresponding standard basis
vector \(e_v\) of \(\mathbb R^V\). For every face \(\sigma\in K\), let
\(|\sigma|:=\conv\{e_v:v\in\sigma\}\). Recall that we denote by \(|\sigma|^\circ\) the relative
interior of this simplex; equivalently, \(|\sigma|^\circ\) consists of those points of \(|\sigma|\) whose support is exactly \(\sigma\). In particular, for every \(v\in V\), we have \(|\{v\}|=|\{v\}|^\circ=\{e_v\}\). The geometric realization \(|K|\) is the disjoint union of
all these \textit{open} simplices:
\[
        |K|=\bigsqcup_{\sigma\in K}|\sigma|^\circ\subset\mathbb R^V.
\]

A finite simplicial complex \(X\) is called a \emph{\(k\)-sphere} if its geometric
realization \(|X|\) is homeomorphic to \(S^k\), that is, there is a continuous
bijection \(|X|\to S^k\) whose inverse is also continuous. We shall use the standard
fact that the join of a \(k_1\)-sphere and a \(k_2\)-sphere is a
\((k_1+k_2+1)\)-sphere.

The following lemma is elementary and can be proved directly. It can also be deduced either from the \(\mathbb Z_2\)-mapping degree argument, see, for example, \cite[Section~2.2]{HatcherAT}, or from the Borsuk--Ulam theorem; see~\cite{Matousek2003}. The \(\mathbb Z_2\)-mapping degree and a generalization of this lemma are discussed in Section~\ref{generalized k-trees}; see Corollary~\ref{degree-cor}.
\begin{lemma}\label{cor:face-covered}
Let \(X\) be a \(k\)-sphere. Then, for an affine map
\(A:|X|\to\mathbb R^k\) and any facet \(\sigma\in X\), we have
\(
        A(\sigma)\subseteq A(X\setminus\{\sigma\}).
\)
\end{lemma}

\section{Proof of Theorem~\ref{theorem main}}
\label{section proof}
For a set \(X\subset\mathbb R^d\), we denote by
\(\Span X\), \(\aff X\), and \(\cone X\) its linear span,
affine hull, and conical hull, respectively. Put
\[
K_0:=V_1*\dots*V_{k+1} \quad\text{and}\quad K_1:=V_{k+2}*\dots *V_{d+1}.
\]
Recall that for any face \(\sigma\in K_0*K_1\), we use the convention \(A(\sigma):=A(|\sigma|^\circ)\), where \(|\sigma|^\circ\) is the relative interior of the simplex \(|\sigma|=\conv\{e_v:v\in \sigma\}\).

For simplicity, we first prove the theorem under the following general-position assumption on the colorful Carath\'eodory map \(A\): For every facet \(\sigma\in K_0*K_1\), the set \(A(\sigma)\) is an open full-dimensional simplex
in \(\mathbb R^d\), and the origin \(0\) does not lie on its boundary. The general case follows from the general-position case by a standard limiting argument. In particular, under this assumption, if \(0\in A(\sigma)\) for \(\sigma\in K_0*K_1\), then \(\sigma\) is a facet.
\smallskip

By the colorful Carath\'eodory theorem, there is a facet \(\sigma\in K_0*K_1\) such that \(0\in A(\sigma)\). If \(\sigma\in T*K_1\), there is nothing to prove. From now on, assume that \(\sigma \notin T*K_1\).

If \(k=d\), then by the definition of a spanning \(k\)-tree, the complex \(T\cup \{\sigma\}\) contains a \(d\)-sphere containing the facet \(\sigma\). By Lemma~\ref{cor:face-covered}, there exists a facet \(\sigma'\) of this \(d\)-sphere distinct from \(\sigma\) such that \(0\in A(\sigma')\), which completes the proof. From now on, we assume that $k<d$. 

The facet \(\sigma\) can be written uniquely as
\(\sigma=\omega\cup\tau\), where \(\omega\) and \(\tau\) are facets of \(K_0\) and \(K_1\), respectively, and \(\omega\notin T\). Put
\[
\mathcal D:=
\bigl\{\delta\text{ is a facet of } K_1:
        0\in A(\omega\cup\delta)\bigr\}.
\]
The set \(\mathcal D\) is nonempty, since it contains \(\tau\). Our next goal is to choose a maximal element of \(\mathcal D\) in a suitable sense.

For every \(\delta\in\mathcal D\), the origin \(0\) lies in \(A(\omega\cup \delta)\), and so the intersection
\[
(-\cone A(\omega))\cap  A(\delta)
\]
is nonempty; see Figure~\ref{fig:x-delta}. By the general-position assumption, it is a singleton; let \(x_\delta\) be its only point. Since \(x_\delta\) lies in \(A(\delta)\), it belongs to the boundary of \(A(\omega\cup \delta)\). Since \(0\in A(\omega\cup \delta)\) does not lie on this boundary, we conclude that \(x_\delta\ne 0\).

Since the dimension of \(\omega\) is \(k<d\), the origin \(0\) does not lie in the \(k\)-dimensional plane \(\aff A(\omega)\). Hence, there is a linear functional
\[
\ell:\Span A(\omega)\to\mathbb R
\quad\text{such that}\quad
\ell|_{\aff A(\omega)}=\mathrm{const}>0.
\]
Since, for every \(\delta\in \mathcal D\), the point \(x_\delta\ne 0\) lies in \(-\cone A(\omega)\subset \Span A(\omega)\), we have
\( \ell(x_\delta)<0\).

\begin{figure}[ht]
\centering
\begin{tikzpicture}[scale=1.2]

    \fill[gray!25]
    (2,1) -- (0,-1) -- (-2,0) -- cycle;

  \fill (0,0) circle (1.4pt) node[below] {$0$};

  \fill (2,1) circle (1.6pt) node[above right] {$A(\omega)$};

  \draw[thick] (0,0) -- (-3,-1.5) node[right] {$-\cone A(\omega)$};

  \draw[thick,dashed] (2,1) -- (0,0);

  \draw[thick,black!75] (0,-1) -- (-2,0) node[below] {$A(\delta)$};

  \fill (-1,-0.5) circle (1.4pt) node[above] {$x_\delta$};

\end{tikzpicture}
\caption{The point \(x_\delta\) is the unique point of
\((-\cone A(\omega))\cap A(\delta)\).}
\label{fig:x-delta}
\end{figure}
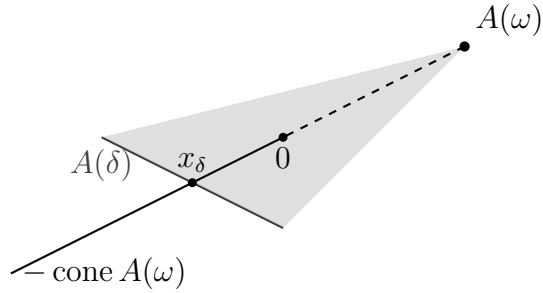

We now choose a maximal element of \(\mathcal D\) in the following sense. Replacing \(\tau\) by another element of \(\mathcal D\), if necessary, we may assume that
\begin{equation}
\label{eq:minimality}
\ell(x_\delta) \leq \ell(x_\tau)<0
        \qquad\text{for all }\delta\in\mathcal D.
\end{equation}

There is a unique linear functional \(h:\mathbb R^d\to\mathbb R\) such that
\[
        h|_{\Span A(\omega)}=\ell
        \qquad\text{and}\qquad
        h|_{\aff A(\tau)}=\mathrm{const}=\ell(x_\tau)<0.
\]
\begin{figure}[ht]
\centering
\begin{tikzpicture}[scale=1.2]

  \fill (0,0) circle (1.4pt) node[below] {$0$};

  \fill (2,1) circle (1.6pt) node[above] {$A(\omega)$};

  \draw[thick] (3,1.5) -- (-3,-1.5) node[below] {$\Span A(\omega)$};

    \draw[thick,black!60, dashed] (-3,0.5) -- (1,-1.5) node[below] {$h(x)=\ell(x_\tau)<0$};
  \draw[thick] (0,-1) -- (-2,0) node[below] {$A(\tau)$};
  \fill (-2,0) circle (1pt) node[above] {$u_{k+2}$};
  \fill (0,-1) circle (1pt) node[above] {$u_{d+1}$};

  \fill (-0.2,1) circle (1pt) node[above] {$v_{k+2}$};
  \fill (1,-0.3) circle (1pt) node[above] {$v_{d+1}$};

  \fill (-1,-0.5) circle (1.4pt) node[above] {$x_\tau$};

\end{tikzpicture}
\caption{Arrangement of the points \(u_{k+2},\dots, u_{d+1}\) and \(v_{k+2},\dots, v_{d+1}\).}
\label{fig:h}
\end{figure}
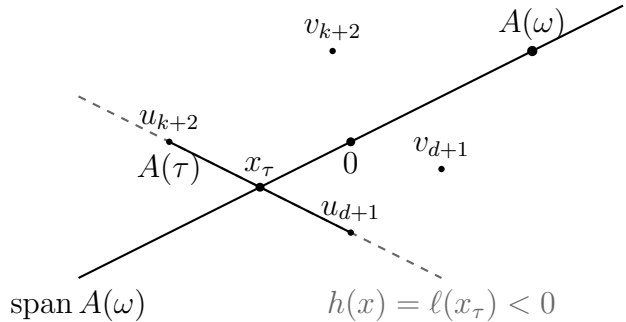

Put
\(
\tau=\{u_{k+2},\ldots,u_{d+1}\},\) where \(u_i\in V_i
\).
Since \(A(u_i)\in \aff A(\tau)\), we have \(h(A(u_i))=\ell(x_\tau)<0\).

Since the origin \(0\) lies in \(A(V_i)\) for every \(i=k+2,
\ldots,d+1\), there is a vertex \(v_i\in V_i\) such that
\(
        h(A(v_i))> \ell(x_\tau)
\); see Figure~\ref{fig:h}. Indeed, otherwise, the set \(A(V_i)\) lies in the closed half-space \(\{x: h(x)\leq \ell(x_\tau)\}\), which does not contain the origin \(0\) since \(\ell(x_\tau)<0\), a contradiction.

By Definition~\ref{definition k-tree}, there is a \(k\)-sphere \(S\subseteq T\cup \{\omega\}\) such that \(\omega\in S\). Since \(h(A(v_i))>\ell(x_\tau)=h(A(u_i))\) for all \(i\), the points \(u_i\) and \(v_i\) are distinct. Hence, the following join of spheres
\[
S*\{u_{k+2},v_{k+2}\}*\dots *\{u_{d+1},v_{d+1}\}
\]
is a \(d\)-sphere containing \(\sigma=\omega\cup \tau\) as a facet. By Lemma~\ref{cor:face-covered} and \(0\in A(\sigma)\), there is a facet \(\sigma'\ne \sigma\) of this sphere with \(0\in A(\sigma')\). Let \(\sigma'=\omega'\cup \tau'\), where \(\omega'\) and \(\tau'\) are facets of \(S\) and of \(\{u_{k+2},v_{k+2}\}*\dots *\{u_{d+1},v_{d+1}\}\), respectively.

If \(\omega'\ne\omega\), then \(\omega'\in T\), and so \(\sigma'=\omega'\cup \tau'\in T*K_1\) is a desired facet. Hence, we can assume that \(\omega'=\omega\). Since \(0\in A(\omega\cup \tau')\), we have \(\tau'\in \mathcal D\). Moreover, \(\tau'\) is distinct from \(\tau\), as \(\sigma'\) is distinct from \(\sigma\).

We claim that \(h(x_{\tau'})>h(x_\tau)\). This contradicts~\eqref{eq:minimality}: indeed, the linear functional \(h\) coincides with the linear functional \(\ell\) on \(\Span A(\omega)\), and the points \(x_\tau\) and \(x_{\tau'}\) lie in \(-\cone A(\omega)\subset \Span A(\omega)\). It remains to prove the claim.

Let \(\tau'=\{w_{k+2},\dots,w_{d+1}\}\), where \(w_i\in \{u_i,v_i\}\) for every \(i\). Since \(\tau'\ne \tau\), at least one of the vertices \(w_i\) coincides with \(v_i\). For all \(i\), we have
\[
        h(A(w_i))\ge h(A(u_i))=h(x_\tau),
\]
and for at least one \(i\) the inequality is strict. Recall that \(A(\tau')\) is the image of the relative interior of \(\conv\{e_v:v\in \tau'\}\). Hence the point \(x_{\tau'}\in A(\tau')\) is a convex combination of the points \(A(w_i)\) with all coefficients positive, and therefore
\(
        h(x_{\tau'})>h(x_\tau)
\), which completes the proof.

\section{Generalized spanning $k$-trees}
\label{generalized k-trees}

The notion of a spanning $k$-tree is not restricted to joins of color classes or wedges of spheres. It has a natural homological formulation for arbitrary finite simplicial complexes, and spanning trees in this sense have been studied, for example, in \cite{Matrix-tree}. However, to the best of our knowledge, the relation between affine maps and spanning $k$-trees has not been considered before. Before giving a formal definition and the corresponding result, we recall the necessary terminology and a few basic properties of the \(\mathbb Z_2\)-mapping degree. For background on the mapping degree, see Section~2.2 of~\cite{HatcherAT}. 

\subsection{Topological preliminaries}

Throughout this subsection, let \(X\) be an abstract \(k\)-dimensional simplicial complex.

All chain groups and homology groups below are taken with coefficients in
\(\mathbb Z_2\). For \(n\geq 0\), we write
\(C_n(X;\mathbb Z_2)\) and \(\widetilde C_n(X;\mathbb Z_2)\) for the ordinary
and reduced \(n\)th simplicial chain groups, respectively, and denote the
boundary operator by \(\partial\). A chain
\(c\in \widetilde C_n(X;\mathbb Z_2)\) is called a \textit{cycle} if
\(\partial c=0\).

If \(B\subseteq X\) is a subcomplex, we write
\(C_n(X,B;\mathbb Z_2)=C_n(X;\mathbb Z_2)/C_n(B;\mathbb Z_2)\) for the relative
\(n\)th chain group. A chain \(c\in C_n(X;\mathbb Z_2)\) is called a \textit{relative
cycle} in the pair \((X,B)\) if
\(\partial c\in C_{n-1}(B;\mathbb Z_2)\). Equivalently, the image of \(c\) in
\(C_n(X,B;\mathbb Z_2)\) is a cycle.

The corresponding ordinary and reduced homology groups are denoted by
\(H_n(X;\mathbb Z_2)\) and \(\widetilde H_n(X;\mathbb Z_2)\). If
\(B\subseteq X\) is a subcomplex, we write \(H_n(X,B;\mathbb Z_2)\) for the
ordinary relative homology group. For relative homology, ordinary and reduced homology agree, and we identify 
\(
\widetilde H_n(X,B;\mathbb Z_2)=H_n(X,B;\mathbb Z_2). 
\) 
This convention is also used when \(B=\varnothing\). Thus 
\[
\widetilde H_n(X,\varnothing;\mathbb Z_2) = H_n(X,\varnothing;\mathbb Z_2) = H_n(X;\mathbb Z_2). 
\] 
In particular, \( \widetilde H_0(X_0,\varnothing;\mathbb Z_2) = H_0(X_0;\mathbb Z_2) = \mathbb Z_2\), where $X_0$ is a singleton.

 Let \(A\colon |X|\to\mathbb R^k\) be an affine map. A point \(y\in\mathbb R^k\) is called a \emph{regular value} of \(A\) if every point of \(A^{-1}(y)\) lies in an open simplex \(|\sigma|^\circ\) for some \(k\)-face \(\sigma\in X\). For a regular value \(y\in\mathbb R^k\), the \(\mathbb{Z}_2\)-mapping degree of \(A\) at \(y\), denoted by \(\deg_2(A,y)\), is defined by
\[
\deg_2(A,y):=\#\{\sigma \in X: y \in A(\sigma)\} \pmod{2}.
\]
This degree has the following homological interpretation. Let \([X] \in \widetilde C_k(X;\mathbb Z_2) \) be the sum of all \(k\)-faces of \(X^{}\). Since \(y\) is regular, the chain \(A_\#([X])\) is a relative cycle in the pair \((\mathbb R^k,\mathbb R^k\setminus\{y\})\). It is well-known that its relative homology class lies in \(H_k(\mathbb{R}^k, \mathbb{R}^k \setminus \{y\}; \mathbb{Z}_2) \cong \mathbb{Z}_2\) and is equal to \(\deg_2(A,y)\). 

\begin{lemma}
\label{degree}
If the chain \([X]\in\widetilde C_k(X;\mathbb Z_2)\) is a cycle, then for any affine map \(A\colon |X|\to\mathbb R^k\) and any regular value \(y\in\mathbb R^k\), we have \(\deg_2(A,y)=0\).
\end{lemma}
\begin{proof}[Sketch of the proof] Since \(\partial [X]=0\), the chain \([X]\) defines a homology class in \(\widetilde H_k(X;\mathbb Z_2)\). Its image in \(\widetilde H_k(\mathbb R^k;\mathbb Z_2)\) is zero, because \(\mathbb R^k\) is contractible. Hence the corresponding relative class in \(H_k(\mathbb R^k,\mathbb R^k\setminus\{y\};\mathbb Z_2) \) vanishes, which implies \(\deg_2(A,y)=0\). \end{proof}

\begin{corollary}
\label{degree-cor}
Under the assumptions of Lemma~\ref{degree}, for any \(k\)-face \(\sigma\in X^{}\), we have
\(
A(\sigma) \subseteq A(X \setminus \{\sigma\}).
\)
\end{corollary}
\begin{proof}[Sketch of the proof]
Suppose the contrary. Since \(A(\sigma)\) is not covered by the image of the boundary of $\sigma$, the set \(A(\sigma)\) is a full dimensional open simplex, and one can choose a point in \(A(\sigma)\) such that it is a regular value and is covered only by $\sigma$. Therefore, we have \(\deg_2(A,y)=1\), which contradicts Lemma~\ref{degree}.
\end{proof}

\subsection{Homological spanning \(k\)-trees}

We now give the homological definition of a spanning \(k\)-tree in an arbitrary finite simplicial complex.

\begin{definition}
\label{SpanningTree}
Let \(X\) be a finite \(k\)-dimensional simplicial complex. A subcomplex
\(T\subseteq X\) is called a \(\mathbb Z_2\)-spanning \(k\)-tree if \(T\)
contains the full \((k-1)\)-skeleton of \(X\) and satisfies:
\begin{enumerate}
    \item \(\widetilde H_k(T;\mathbb Z_2)=0\);
    \item $T$ is maximal with respect to (1), that is,
    \(
    \widetilde H_k(T\cup\{\sigma\};\mathbb Z_2)\ne 0
    \)
    for every facet \(\sigma\) of \(X\) that does not belong to \(T\). 
\end{enumerate}
\end{definition}

The following lemma follows directly from Corollary~\ref{degree-cor} and from the definition of the $\mathbb{Z}_2$-spanning $k$-tree.

\begin{lemma}
\label{Z_2k-tree}    
Let \(X\) be a finite \(k\)-dimensional simplicial complex, and let \(T\subseteq X\) be a \(\mathbb Z_2\)-spanning \(k\)-tree. Then, for every affine map \(A:|X|\to\mathbb R^k\), we have \(A(T)=A(X)\).
\end{lemma}

Throughout the rest of this subsection, we use the notation and conventions from the introduction.

 Theorem~\ref{theorem main} remains valid if the spanning \(k\)-tree \(T\) is replaced by a \(\mathbb Z_2\)-spanning \(k\)-tree. We state this more general version explicitly.

\begin{theorem}
\label{general theorem main}
Let \(T\subseteq V_1*\cdots *V_{k+1}\) be a \(\mathbb Z_2\)-spanning \(k\)-tree for some \(0\leq k\leq d\). If \(A:|\Delta_V|\to\mathbb R^d\) is a colorful Carath\'eodory map, then there exists a face \( \sigma\in T*V_{k+2}*\cdots *V_{d+1} \) such that \(0\in A(\sigma)\). \end{theorem}

\begin{proof}[Sketch of proof] 
The proof is the same as that of Theorem~\ref{theorem main}. The only change is that the \(k\)-sphere \(S\) used there is replaced by the support \(C\) of the relevant cycle in \(T\cup\{\omega\}\). Since the join of cycles is again a cycle, the complex \[ C*\{u_{k+2},v_{k+2}\}*\cdots *\{u_{d+1},v_{d+1}\} \] induces a \(d\)-dimensional cycle. The geometric argument remains the same. \end{proof}

\section{Discussion}
\label{section discussion}
In this section we discuss the difference between Theorems \ref{theorem main} and \ref{general theorem main}. Then we explore applications of Theorem \ref{general theorem main} and, in particular,  Lemma~\ref{Z_2k-tree}, to some other Caratheodory type theorems.

\subsection{Spanning $k$-tree vs $\mathbb{Z}_2$-spanning $k$-tree}

The homological notion of a \(\mathbb Z_2\)-spanning \(k\)-tree is genuinely broader than the spherical notion used in Theorem~\ref{theorem main}. Already in dimension \(2\), the standard triangulation of the torus, with vertices partitioned as \(V_1\sqcup V_2\sqcup V_3\) as in Figure~\ref{fig:k333-torus}, gives a subcomplex of \(V_1*V_2*V_3\). Its fundamental cycle is not supported on a triangulated sphere, so the homological version allows \(k\)-trees that are not covered by the spherical definition used in Theorem~\ref{theorem main}. 



\begin{figure}[ht]
\centering
\begin{tikzpicture}[
    scale=1.1,
    vertex/.style={circle, draw=black, inner sep=0pt, minimum size=5.5pt},
    vone/.style={vertex, fill=blue!70},
    vtwo/.style={vertex, fill=green!65!black},
    vthree/.style={vertex, fill=red!75}
]

\draw[thick] (0,0) rectangle (3,3);

\foreach \x in {1,2} {
    \draw[thick] (\x,0) -- (\x,3);
}
\foreach \y in {1,2} {
    \draw[thick] (0,\y) -- (3,\y);
}

\foreach \x in {0,1,2} {
    \foreach \y in {0,1,2} {
        \draw[thick] (\x,\y) -- (\x+1,\y+1);
    }
}


\node[vone]   at (0,0) {};
\node[vtwo]   at (1,0) {};
\node[vthree] at (2,0) {};
\node[vone]   at (3,0) {};

\node[vtwo]   at (0,1) {};
\node[vthree] at (1,1) {};
\node[vone]   at (2,1) {};
\node[vtwo]   at (3,1) {};

\node[vthree] at (0,2) {};
\node[vone]   at (1,2) {};
\node[vtwo]   at (2,2) {};
\node[vthree] at (3,2) {};

\node[vone]   at (0,3) {};
\node[vtwo]   at (1,3) {};
\node[vthree] at (2,3) {};
\node[vone]   at (3,3) {};

\node[below] at (1.5,-0.15) {\small bottom};
\node[above] at (1.5,3.15) {\small top};
\node[left]  at (-0.15,1.5) {\small left};
\node[right] at (3.15,1.5) {\small right};

\end{tikzpicture}
\caption{Torus triangulation. Opposite sides are identified, and the vertex colors represent the three parts \(V_1,V_2,V_3\).}
\label{fig:k333-torus}
\end{figure}
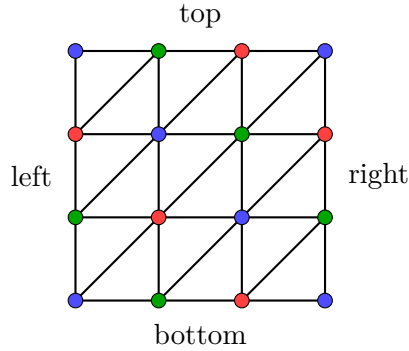

Theorem \ref{general theorem main}, and especially Lemma~\ref{Z_2k-tree}, allows us to refine some well-known generalizations of the Carath\'eodory theorem. 

\subsection{Very colorful Carath\'eodory theorem}
Recall that the very colorful Carath\'eodory theorem~\cite{HolmsenPachTverberg2008,ArochaBaranyBrachoFabilaMontejano2009} asserts that, for \(d+1\) non-empty finite sets \(V_1,\ldots,V_{d+1}\subset\mathbb R^d\), if the origin lies in the convex hull of \(V_i\cup V_j\) for every pair \(i<j\), then there are points \(v_i\in V_i\) such that \(0\in\conv\{v_1,\ldots,v_{d+1}\}\). Lemma~\ref{Z_2k-tree} gives the following refinement.


\begin{theorem}
\label{thm:very-colorful-tree}
Let \(V_1,\ldots,V_{d+1}\) be pairwise disjoint non-empty finite sets, put \(V=V_1\sqcup\cdots\sqcup V_{d+1}\), and let \( T\subseteq V_1*\cdots*V_{d+1} \)
be a \(\mathbb Z_2\)-spanning \(d\)-tree. Let \(A:|\Delta_V|\to\mathbb R^d\) be an affine map such that \( 0\in A(V_i\cup V_j) \)
for every pair \(1\leq i<j\leq d+1\). Then there exists a face \(\sigma\in T\) such that \(0\in A(\sigma)\).
\end{theorem}

Indeed, the classical very colorful Carath\'eodory theorem gives
\(0\in A(V_1*\cdots*V_{d+1})\), and Lemma~\ref{Z_2k-tree} then implies \(0\in A(T)\).

However, Theorem~\ref{thm:very-colorful-tree} does not extend in this form to spanning \(k\)-trees with \(k<d\). Already the case \(k=d-1\) fails. For simplicity, we describe the example by identifying each vertex with its image under the affine map. Let \(d\ge 2\) and let \(e_1,\ldots,e_d\) be the standard basis of \(\mathbb R^d\). Put \(V_i=\{e_i,-e_i\}\) for \(i=1,\ldots,d\), and let \(V_{d+1}=\{w\}\), where \(w=\varepsilon(e_1+\cdots+e_d)\) for sufficiently small \(\varepsilon>0\). 

Let \(\omega_-=\{-e_1,\ldots,-e_d\}\), and let \(T\subset V_1*\cdots*V_d\) be obtained by deleting the facet \(\omega_-\) and keeping all its proper faces. Then \(T\) is a \(\mathbb Z_2\)-spanning \((d-1)\)-tree in \(V_1*\cdots*V_d\). It is straightforward to check that the pair-color assumptions are satisfied, while no face \(\sigma\in T*V_{d+1}\) satisfies \(0\in A(\sigma)\).

\subsection{Matroidal variation}

We also record a matroidal refinement of the colorful Carath\'eodory theorem of Kalai and Meshulam~\cite{KalaiMeshulam}.

Recall that if \(K\) is a simplicial complex with vertex set \(V\) and
\(W\subseteq V\), then \(K[W]\) denotes the induced subcomplex consisting of all faces of \(K\) contained in \(W\). A simplicial complex is called \emph{pure} if all its facets have the same dimension. We say that \(M\subset \Delta_V\) is \emph{matroidal} if \(M[W]\) is pure for every \(W\subseteq V\). Its rank function is \(\rho(W)=\dim M[W]+1,\) with the convention \(\dim\{\varnothing\}=-1\). When \(M\) has rank \(d+1\), we call an inclusion-minimal subset \(S\subseteq V\) satisfying \(\rho(V\setminus S)\le d\) a \emph{cocircuit} of \(M\).

\begin{theorem}
Let \(M\subset \Delta_V\) be a matroidal complex of rank \(d+1\), with rank
function \(\rho\), and let \(T\subset M\) be a \(\mathbb Z_2\)-spanning \(d\)-tree. Let \(A\colon |\Delta_V|\to\mathbb R^d\) be an affine map such that \( 0\in A(S) \) for every cocircuit \(S\) of \(M\). Then there exists a face
\(\sigma\in T\) such that \(0\in A(\sigma)\).
\end{theorem}

Indeed, the matroidal colorful Carath\'eodory theorem of Kalai and Meshulam~\cite{KalaiMeshulam} gives \(0\in A(M)\), and by Lemma~\ref{Z_2k-tree} there exists a face \(\sigma\in T\) such that \(0\in A(\sigma)\).







\subsection{The Tverberg theorem}

We finally mention one further consequence, without formulating it as a separate theorem. Recall that the affine Tverberg theorem~\cite{Tverberg1966} asserts that every affine map from a simplex of dimension \((r-1)(d+1)\) to \(\mathbb R^d\) maps \(r\) pairwise disjoint faces to sets with a common point. In Corollary~1.7 of~\cite{Blag}, Blagojevi\'c proved a constrained version of this statement based on Theorem~\ref{theorem blagojevic k=1}. The same argument gives the corresponding statement with Theorem~\ref{theorem main} in place of Theorem~\ref{theorem blagojevic k=1}; moreover, the same extension works with the more general Theorem~\ref{general theorem main}. Since no new ideas are involved, we only record this observation and refer the reader to Blagojevi\'c's paper for the original formulation.


\section*{Acknowledgments}
We thank Roman Karasev for helpful discussions and comments, including the suggestion to use the homological definition of a spanning tree. We thank Andreas Holmsen for suggesting that we write the appendix. We are also grateful to Alexander Golovanov and Andrey Ryabichev for reading an earlier draft of the paper and for their remarks, which helped improve the exposition. We are also grateful to the participants of the MIPT seminar ``Combinatorics and Topology'' for helpful discussions.

\section*{Statements and Declarations}

\subsection*{Funding}
The research of M.B. is supported by MSHE RF GZ project. The research of A.P. is partially supported by NSF Grant DMS 2349045. 

\subsection*{Competing Interests}
The authors have no competing interests to declare that are relevant
to the content of this article.

\subsection*{Data Availability}
No datasets were generated or analysed during the current study.

\subsection*{Use of Generative AI}
The authors used generative AI tools to assist in verifying the proofs and improving the exposition and clarity. Definition~\ref{definition:meshulam-admissible} and the proof of Theorem~\ref{theorem:meshulam-admissible} from the appendix were suggested by OpenAI Codex using the GPT-5.5 model.  The authors independently verified the AI-assisted material and take full responsibility for the content of the manuscript.

\printbibliography

\appendix
\section{An alternative proof via Meshulam's lemma}
\label{appendix meshulam proof}

In this appendix, we provide an alternative proof of Theorem~\ref{general theorem main} using Meshulam's lemma.

\subsection{Preliminaries}
\label{appendix preliminaries}

We use the notation and conventions from the main body of the paper. For a
positive integer \(m\), write \([m]=\{1,\ldots,m\}\). If \(K\) is a
simplicial complex on a finite ground set \(V\), its \emph{Alexander dual}
(with respect to \(V\)) is
\[
K^*:=\{F\subseteq V:V\setminus F\notin K\}.
\]
When forming the join of two simplicial complexes that are
both defined on \(V\), we regard them as complexes on two disjoint labeled copies \(V^{(1)}\) and \(V^{(2)}\), respectively.
For \(U\subseteq V\), the \emph{induced subcomplex} of \(K\) on \(U\) is
\(K[U]:=\{\sigma\in K:\sigma\subseteq U\}\).

All homology groups in this appendix have coefficients in \(\mathbb Z_2\).
Recall that we use reduced homology, so
\(\widetilde H_{-1}(K;\mathbb Z_2)=0\) if and only if \(K\) has a vertex.
For an integer \(q\), we say that \(K\) is
\emph{homologically \(q\)-connected} if
\(\widetilde H_i(K;\mathbb Z_2)=0\) for every \(-1\leq i\leq q\).
For two simplicial complexes \(K_1\) and \(K_2\), the K\"unneth formula for joins over a field \cite[(9.12)]{inbook} yields
\begin{equation}
\label{eq:join-kunneth}
 \widetilde H_n(K_1*K_2;\mathbb Z_2)
 \cong
 \bigoplus_{a+b=n-1}
 \widetilde H_a(K_1;\mathbb Z_2)
 \otimes_{\mathbb Z_2}
 \widetilde H_b(K_2;\mathbb Z_2).
\end{equation}
Here, the indexes \(a\) and \(b\) in the direct sum are integers from \(\{-1,\dots,n\}\). In particular, the join of a homologically \(p\)-connected complex and a
homologically \(q\)-connected complex is homologically
\((p+q+2)\)-connected.

We shall use the following form of Meshulam's lemma; see
\cite[Proposition~2.6]{DeLoeraGoaocMeunierMustafa2019}. Here and below,
\(\#I\) denotes the cardinality of a finite set \(I\).

\begin{lemma}[Meshulam's lemma]
\label{lemma meshulam}
Let \(K\) be a finite simplicial complex on a ground set \(V\), and let
\(\lambda\colon V\to[m]\) be a labeling. If
\[
 \widetilde H_{\#I-2}\bigl(K[\lambda^{-1}(I)];\mathbb Z_2\bigr)=0
\]
for every nonempty \(I\subseteq[m]\), then \(K\) contains an
\((m-1)\)-dimensional face on which \(\lambda\) is a bijection onto \([m]\).
\end{lemma}

We next recall the definition of zero-avoiding complex. Fix a finite vertex set \(V\), let
\(A\colon|\Delta_V|\to\mathbb R^d\) be affine, and set
\(\mathcal C_A:=\{S\subseteq V:0\notin A(\Delta_S)\}\).
Thus
\begin{equation}
\label{eq:zero-avoiding-dual}
 \mathcal C_A^*
 =
 \bigl\{F\subseteq V:
 0\in A(\Delta_{V\setminus F})\bigr\}.
\end{equation}





The zero-avoiding complex \(\mathcal C_A\) is a special case of the support
complexes of oriented matroids studied by Holmsen in \cite{Holmsen2016}.
Holmsen proved that if the associated oriented matroid has rank at most \(d\),
then its support complex is near-\((d-1)\)-Leray; see
\cite[Proposition~1.3]{Holmsen2016}. Together with
\cite[Lemma~3.1]{Holmsen2016}, this result yields
Lemma~\ref{lemma:zero-avoiding-homology} below. The two assertions of Lemma~\ref{lemma:zero-avoiding-homology} also appear
in Blagojevi\'c's work: assertion~(1) follows from
\cite[Lemma~2.11]{Blag}, while assertion~(2) is \cite[Lemma~2.15]{Blag}.

\begin{lemma}
\label{lemma:zero-avoiding-homology}
Let \(A\colon|\Delta_V|\to\mathbb R^d\) be an affine map.
\begin{enumerate}

\item \(\widetilde H_i(\mathcal C_A^*;\mathbb Z_2)
  =0\) for every \(i\leq \#V-d-3\).

 \item For every proper subset \(U\subsetneq V\),
 \(\widetilde H_i(\mathcal C_A^*[U];\mathbb Z_2)=0\) whenever
 \(i\leq \#U-d-2\).
\end{enumerate}
\end{lemma}


\subsection{Meshulam-admissible complexes}

\begin{definition}
\label{definition:meshulam-admissible}
Let \(m\geq1\), let \(V_1,\ldots,V_m\) be pairwise disjoint finite sets, and
put \(V=V_1\sqcup\cdots\sqcup V_m\). A subcomplex
\(J\subseteq V_1*\cdots*V_m\), regarded as a complex on the ground set \(V\),
is called \emph{Meshulam-admissible} if:
\begin{enumerate}
 \item \(J\) is homologically \((m-2)\)-connected;
 \item for every proper subset \(U\subsetneq V\) satisfying
 \(U\cap V_i\ne\varnothing\) for all \(i\in[m]\), the induced subcomplex
 \(J[U]\) is homologically \((m-3)\)-connected.
\end{enumerate}
\end{definition}


\begin{theorem}
\label{theorem:meshulam-admissible}
Let \(K\subseteq V_1*\cdots*V_{d+1}\) be Meshulam-admissible, and let
\(A\colon|\Delta_V|\to\mathbb R^d\) be a colorful Carath\'eodory map.
Then there exists a face \(\sigma\in K\) such that \(0\in A(\sigma)\).
\end{theorem}

\begin{proof}
Consider the complex \(\mathcal C_A^* * K\) on
\(V^{(1)}\sqcup V^{(2)}\), where the first factor is supported on
\(V^{(1)}\) and the second on \(V^{(2)}\). Let \(m:=\#V\), identify \(V\)
with \([m]\), and label the vertices of \(\mathcal C_A^* *K\) by
\(\lambda(v^{(1)})=\lambda(v^{(2)})=v\).

For every nonempty \(U\subseteq V\), we have \( (\mathcal C_A^*K)[\lambda^{-1}(U)]=\mathcal C_A^*[U]*K[U]. \)
Let us show that one can apply Lemma~\ref{lemma meshulam} to $\mathcal C_A^* *K$. For that, it is enough to show
\[
\widetilde H_{\#U-2}
\bigl(\mathcal C_A^*[U]*K[U];\mathbb Z_2\bigr)=0.
\]
By the K\"unneth formula, it remains to prove that
\begin{equation}
\label{equation direct sum is 0}
 \widetilde H_a(\mathcal C_A^*[U];\mathbb Z_2)
 \otimes_{\mathbb Z_2}
 \widetilde H_b(K[U];\mathbb Z_2)=0
\end{equation}
for all \(a,b\in\{-1,\ldots,\#U-2\}\) with
\(a+b=\#U-3\). We distinguish three cases.

First, suppose that \(U=V\). If the first factor
in~\eqref{equation direct sum is 0} is nonzero, then
Lemma~\ref{lemma:zero-avoiding-homology} gives
\(a\geq\#V-d-2\). Hence \(b\leq d-1\), so the second factor
in~\eqref{equation direct sum is 0} vanishes by condition~(1)
of Definition~\ref{definition:meshulam-admissible}.

Next, suppose that \(U\subsetneq V\) and \(U\cap V_i=\varnothing\) for some
\(i\). Then \(V_i\subseteq V\setminus U\), and the colorful assumption
gives \(0\in A(V_i)\subseteq A(\Delta_{V_i})\subseteq A(\Delta_{V\setminus U})\).
Consequently,
\(\mathcal C_A^*[U]=\Delta_U\), which means that
\(\mathcal C_A^*[U]\) is contractible as well as \(\mathcal C_A^*[U]*K[U]\).

Finally, suppose that \(U\subsetneq V\) meets every color class. If the first
factor in~\eqref{equation direct sum is 0} is nonzero,
Lemma~\ref{lemma:zero-avoiding-homology} gives
\(a\geq\#U-d-1\). Therefore,
\(b=\#U-3-a\leq d-2\), and the second factor in~\eqref{equation direct sum is 0} vanishes by condition~(2) of
Definition~\ref{definition:meshulam-admissible}.

Lemma~\ref{lemma meshulam} now gives a fully labeled face
\(F^{(1)}\cup G^{(2)}\) of \(\mathcal C_A^* *K \), where \(F\in\mathcal C_A^*\), \(G\in K\),
and \(F\sqcup G=V\). Since \(F\in\mathcal C_A^*\), we have
\(0\in A(\Delta_{V\setminus F})=A(\Delta_G)\). Choose an inclusion-minimal face
\(\sigma\subseteq G\) such that \(0\in A(\sigma)\).
\end{proof}

\begin{lemma}
Let \(J\subseteq V_1*\cdots*V_m\) be Meshulam-admissible, and let
\(V_{m+1}\) be a nonempty finite set disjoint from
\(V=V_1\sqcup\cdots\sqcup V_m\). Then \(J*V_{m+1}\) is
Meshulam-admissible with respect to \(V_1,\ldots,V_{m+1}\).
\end{lemma}

\begin{proof}
The complex \(V_{m+1}\), as well as each of its nonempty induced
subcomplexes, is homologically \((-1)\)-connected. Hence \(J*V_{m+1}\) is
homologically \((m-1)\)-connected. Now let
\(U\subsetneq V\sqcup V_{m+1}\) meet every color class. Then
\(J[U\cap V]\) is homologically \((m-3)\)-connected, so
\(J[U\cap V]*(U\cap V_{m+1})\) is homologically
\((m-2)\)-connected. Thus both conditions in
Definition~\ref{definition:meshulam-admissible} hold.
\end{proof}

\begin{corollary}
\label{corollary:meshulam-spanning-tree}
Let \(T\subseteq V_1*\cdots*V_{k+1}\) be a
\(\mathbb Z_2\)-spanning \(k\)-tree, where \(0\leq k\leq d\). Then
\(K:=T*V_{k+2}*\cdots*V_{d+1}\) is Meshulam-admissible. Consequently,
Theorem~\ref{theorem:meshulam-admissible} implies
Theorem~\ref{general theorem main}.
\end{corollary}

\begin{proof}
First, \(T\) is Meshulam-admissible with respect to
\(V_1,\ldots,V_{k+1}\). Indeed, for \(k\geq1\), the full-skeleton condition
gives the required homology vanishing through degree \(k-2\), while the
two conditions in Definition~\ref{SpanningTree} imply the vanishing in
degree \(k-1\). For every proper subset meeting all \(k+1\) color classes, the
corresponding induced subcomplex of \(T\) contains the full
\((k-1)\)-skeleton and is therefore homologically
\((k-2)\)-connected. The case \(k=0\) is immediate. Repeated application of the join
lemma, with \(V_{k+2},\ldots,V_{d+1}\), now proves the result.
\end{proof}

\end{document}